\documentstyle{llncs}
\input psfig.sty

\makeindex

\newcommand{\EMPTYSET}{\mbox{$\O$}}

\newcommand{\HHS}{\mbox{$\hspace{2pt}$}}
\newcommand{\CL}{\mbox{$\varphi$}}
\newcommand{\NCL}{\mbox{$\varphi_{\eta}$}}

\newcommand{\CALA}{\mbox{${\cal A}$}}

\newcommand{\CALS}{\mbox{${\cal S}$}}

\newcommand{\COMP}{\mbox{${\HHS \bf \cdot} \HHS$}}

\newcommand{\FTRANS}{\mbox{${\ \stackrel{f}{\longrightarrow}\ }$}}

\newcommand{\GTRANS}{\mbox{${\ \stackrel{g}{\longrightarrow}\ }$}}
\newcommand{\FGTRANS}{\mbox{${\ \stackrel{f.g}{\longrightarrow}\ }$}}

\newcommand{\NBHD}{\mbox{$\eta$}}
\newcommand{\REG}{\mbox{$\rho$}}

\begin{document}

\bibliographystyle{plain}
\bibstyle{plain}

\title{
	{\bf Entropy in Social Networks}
	}
   
\author{
	John L. Pfaltz
 	}
\institute{
	Dept. of Computer Science,  University of Virginia \\
	{\tt jlp@virginia.edu } \\
    }

\maketitle

\begin{abstract}
We introduce the concepts of closed sets and closure operators
as mathematical tools for the study of social networks.
Dynamic networks are represented by transformations.

It is shown that under continuous change/transformation, all
networks tend to ``break down'' and become less complex.
It is a kind of entropy.

The product of this theoretical decomposition is an abundance
of triadically closed clusters which sociologists have 
observed in practice.
This gives credence to the relevance of this kind of 
mathematical analysis in the sociological context.
\end{abstract}
\section {Introduction} \label{I}

The term ``entropy'' has a variety of interpretations depending on its
context.
In information theory, it is a measure of ``surprise'', the difference 
between the actual transmission and a purely random signal, or
Shannon entropy.
This idea of defining network ``complexity'' as the difference between 
a given network and a random network has been pursued by several 
researchers, $e.g.$
\cite{AnaBia09,Bia08,Bia09,LiWanWanZho08}.
Two problems  are ``what is a $random$ network?'', and
``how does one measure the $difference$ between networks?''
Not surprisingly, most efforts are statistical, frequently employing
an eigen-analysis of the adjacency matrix $\CALA$ defining the network
\cite{New06,RicSea00}.

A different interpretation of ``entropy'' is associated with the idea
that dynamic systems gradually lose ``energy'' unless maintained by an
outside source.
Here entropy can be viewed as the energy needed to maintain a steady state.
We subscribe to a version of this latter interpretation.

Many social networks are dynamic
\cite{KosWat06,McCCar11}.
One can examine network structure in the light of network change and
ask how the structure is a result of social processes.
Demetrius and Manke
\cite{DemMan05}
is a good example of this approach, although their methodology is rather
different from ours.
Dynamic social processes also play a role in Granovetter's work
\cite{Gra73},
which we will examine more closely in Section \ref{CASN}.

In this paper we describe network structure in terms of closed sets
and closure operators, which we introduce in Section \ref{C}.
Social processes are modeled by mathematical transformations defined
in Section \ref{T}.
We then ask how various kinds of transformations, specifically
continuous transformations, affect closure relations.
Of particular interest is the question: given the concepts of 
separation and connectivity defined in terms of closure operators 
often used in a social context (Section \ref{SC}), what happens
to separated/connected sets under a continuous transformation.
The results are somewhat surprising, and lead to triadically
closed sub-structures that appear to be relevant in social analysis.

\section {Dynamic Closure Spaces} \label{DCS}

Discrete structures, such as networks, are hard to describe in detail.
They may be visualized as graphs
\cite{Fre00};
but as the structures become large this becomes increasingly difficult.
One may use statistics to describe features, $e.g.$
\cite{KleKumRag99,LesLanDasMah08};
but this conveys little information regarding local structure.

We have found the concept of closure and closed sets over a ground
set, $P$, of ``points'' or ``nodes'' or ``individuals'' to be of
value in a variety of discrete applications
\cite{Pfa96,Pfa08,Pfa11}.
In this section, we briefly develop the general mathematics of this
approach.
In Section \ref{CASN}, we specialize these ideas to the case of
social networks.

\subsection{Closure} \label{C}

An operator $\CL$ is said to be a {\bf closure operator}
if for all sets $Y, Z \subseteq P$, it is:
 \\
\hspace*{0.5in}
(C1)   extensive, $Y \subseteq Y.\CL$,
\\
\hspace*{0.5in}
(C2)   monotone, $Y \subseteq Z$ implies $Y.\CL \subseteq Z.\CL$, and,
\\
\hspace*{0.5in}
(C3)  idempotent, $Y.\CL.\CL = Y.\CL$.
\\
A subset $Y$ is closed if $Y = Y.\CL$.
In this work we prefer to use suffix notation, in which an operator follows
its operand.
Consequently, when operators are composed the order of application is
read naturally from left to right.
With this suffix notation read $Y.\CL$ as ``$Y$ closure''.
It is well known that
the intersection of closed sets must be closed;
sometimes this is easier to verify.
This latter can be used as the definition of closure,
with the operator $\CL$ defined by
$Y.\CL = \bigcap_{Z_i \ closed} \{ Y \subseteq Z_i \}$.

By a {\bf closure system} $\CALS = (P, \CL)$, we mean a set $P$ of
``points'', ``elements'' or ``individuals'', together with a closure operator $\CL$.
By (C1) the set $P$ must itself be closed. 
In a social network these points are typically individuals, or possibly
institutions.
A set $Y$ is closed if $Y.\CL = Y$.
The empty set, $\EMPTYSET$, may, or may not, be closed.

\subsection {Transformations} \label{T}

A {\bf transformation}, $f$, is a function that maps the sets of one
closure system $\CALS$ into another $\CALS'$.
Because our usual sense of functions, defined with the ground set $P$ as
their domain, are typically expressed in prefix notation
we use a suffix notation to emphasize the set characteristics of
transformations, particularly reminding us that their domain is 
the power set of $P$, or $2^P$, and their value is always a $set$ in the
range, or codomain.

A transformation $(P, \CL) \FTRANS (P', \CL')$ is said to be {\bf monotone} if
$X \subseteq Y$ in $P$ implies $X.f \subseteq Y.f$ in $P'$.

A transformation $(P, \CL) \FTRANS (P', \CL')$ is said to be {\bf continuous} if
for all sets $Y \in P$, $Y.\CL.f \subseteq Y.f.\CL'$,
\cite{Ore46,PfaSla12,Sla04}.
Proofs of the following two propositions can be found in
\cite{PfaSla12,Pfa11}.

\begin{proposition}\label{p.COMP.CO}
Let $(P, \CL) \FTRANS (P', \CL')$, $(P', \CL') \GTRANS (P'', \CL'')$
be monotone transformations.
If both
$f$ and $g$ are continuous, then so is $(P, \CL) \FGTRANS (P'', \CL'')$.
\end{proposition}

With topological closure over domains of real variables, the inverse image of any
closed set under a continuous function must be closed.
The following proposition provides a discrete analog.

\begin{proposition}\label{p.IC1}
Let $(P, \CL) \FTRANS (P', \CL')$ be monotone, continuous and let
$Y' = Y.f$ be closed.
Then $Y.\CL.f = Y'$.
\end{proposition}

$Y$, itself need not be closed; but its closure $Y.\CL$ must have 
the same closed image.

\begin{proposition}\label{prop.IC4}
Let $(P, \CL) \FTRANS (P', \CL')$ be monotone and continuous.
If $X.\CL = Y.\CL$ then $X.f.\CL' = Y.f.\CL'$.
\end{proposition}
\begin{proof}
Let $f$ be continuous and assume that
$X.\CL = Y.\CL$.
By monotonicity and continuity, $X.f \subseteq X.\CL.f  = Y.\CL.f \subseteq Y.f.\CL'$.
Similarly, $Y.f \subseteq X.\CL.f \subseteq X.f.\CL'$.
Consequently, $X.f$ and $Y.f$ are contained in $X.f.\CL' \cap Y.f.\CL' = X.f.\CL' = Y.f.\CL'$.
\qed
\end{proof}

A transformation $\CALS \FTRANS \CALS'$ is said to be {\bf surjective} if
for every closed $Y' \in \CALS'$ there exists a set $Y \in \CALS$
(not necessarily closed)
such that $Y.f = Y'$.
This definition of surjectivity overcomes one of the curses associated with
transformations over finite spaces.
It allows a smaller space, of lesser cardinality, to map "onto" a larger space.

\begin{samepage}
\begin{proposition}\label{p.IC2}
Let $f$ be monotone, continuous and surjective, then for all closed
$Y'$ in $P'$, there exists a closed $Y$ in $P$ such that
$Y.f = Y'$.
\end{proposition}
\end{samepage}
\noindent
\begin{proof}
Since $f$ is surjective, $\exists Y, Y.f = Y'$.
Since $f$ is continuous, by Prop. \ref{p.IC1}, $Y.\CL.f = Y'$.
\qed
\end{proof}

\begin{proposition}\label{prop.COMP.SUR}
Let $(P, \CL) \FTRANS (P', \CL')$, $(P', \CL') \GTRANS (P'', \CL'')$
be transformations and let $g$ be monotone, continuous.
If both
$f$ and $g$ are surjective, then so is $(P, \CL) \FGTRANS (P'', \CL'')$.
\end{proposition}
\begin{proof}
Because of Prop. \ref{p.COMP.CO} we need only consider surjectivity.
Let $Y''$ be closed in $P''$.  
Since $g$ is surjective, $\exists Y' \in P'$, $Y'.g = Y''$.
Because, $g$ is continuous we may assume, by Prop. \ref{p.IC1}, that
$Y'$ is closed.
Thus, by surjectivity of $f$, $\exists Y \in P, Y.f = Y'$
Consequently,
$f\COMP g$ is surjective.
\qed
\end{proof}

By $Y'.f^{-1}$ we mean the collection $\{ Y | Y.f = Y' \}$.
Even when $f$ is surjective there may be no inverse $Y'.f^{-1}$
unless $Y'$ is closed in $S'$.
\section {Closure Applied to Social Networks} \label{CASN}

A space, $\CALS$ is said to be {\bf atomistic} if for all singleton
sets $\{x\}$ and $\{y\}$, there can be no transformation 
$\CALS \FTRANS \CALS'$ such that $\{x\}.f = \{y\}.f = y' \neq \EMPTYSET$;
or equivalently, $\{x\}.f = \{y\}.f$ implies $\{y\}.f = \EMPTYSET$.

Atomisticity is a characteristic of the elements in the ground set of
the system $\CALS$.
For example, chemical elements are atomistic;  no two elements can
``fuse'' to become a single element, even though they can combine to 
form more complex structures, or molecules.
On the other hand, if the elements of our network are corporations, it
is not unusual to have two corportation merge into one unit so the space
is not atomistic.
In this paper we are concerned with
human beings in a social network who are clearly atomistic. 
No matter how the social network changes, two individuals are still
separate individuals, provided they are still in the network.
Whether chemical elements or social individuals, we may have $\{x\}.f = \{y\}.f = \EMPTYSET$
if both $x$ and $y$ are removed from the system;  but they cannot be combined.

\begin{samepage}
\begin{proposition}\label{p.ATOM}
If $\CALS$ is atomistic, then for all monotone transformations, $f$,
\\
\hspace*{0.2in}
(a)
for all singleton sets, if $\{y\}.f \neq \EMPTYSET$,
$\{y\}.f.f^{-1} = \{y\}$;
\\
\hspace*{0.2in}
(b)
$(X\ \cap\ Z).f = X.f \ \cap\ Z.f$;
\\
\hspace*{0.2in}
(c)
$(X\ \cup\ Z).f = X.f \ \cup\ Z.f$.
\end{proposition}
\end{samepage}
\noindent
\begin{proof}
(a)
Readily, $\{y\} \in \{y\}.f.f^{-1}$.
If $\exists \{x\} \neq \{y\} \in \{y\}.f.f^{-1}$ then
$\{x\}.f = \{y\}.f = y'$ violating atomicity.
\\
(b)
Monotonicity ensures that $(X \cap Z).f \subseteq X.f \cap Z.f$.
\\
Let $y' \in X.f \cap Z.f$.
By (a) $y' \in X.f$ implies $y = y'.f^{-1} \in X$.
Similarly, $y' \in Z.f$ implies $y = y'.f^{-1} \in X \cap Z$.
So $X.f \cap Z.f \subseteq (X \cap Z).f$.
\\
(c)
Again, monotonicity ensures $X.f \cup Z.f \subseteq (X \cup Z).f$
And atomicity ensures $(X \cup Z).f \subseteq X.f \cup Z.f$ in
the same manner as (b) above.
\qed
\end{proof}

Even though atomicity appears to be a natural, real world constraint;
its mathematical consequences are considerable.
Effectively, any transformation $f$ defined over an atomistic domain is
the identity map;
only the relationship structures between sets can be altered.
However, one can have $Y.f = \EMPTYSET'$ and $\EMPTYSET.f = Y'$.\footnote
	{
	Normally, we do not distinguish between $\EMPTYSET$ and
	$\EMPTYSET'$.
	The empty set is the empty set. 
	We do so here only for emphasis.
	}
And, monotonicity ensures that for all $X \subseteq Y$, $X.f = \EMPTYSET'$,
and for all $Z' \subseteq Y'$, $\EMPTYSET.f = Z'$, so $Z'.f^{-1} = \EMPTYSET$.

\subsection{Neighborhood Closure} \label{NC}

A closure operator that seems particularly appropriate in the social network context
is the ``neighborhood closure'' because ``neighborhoods'' can play a central role
in social behavior
\cite{HipFarBoe12}.
Let
$\CALS = (P, \CALA)$ be a set $P$ of points,
or elements, together with a symmetric adjacency 
relation $\CALA$.
By the {\bf neighborhood}, or neighbors, of a set $Y$ we mean the
set $Y.\NBHD = \{ z \not\in Y | \exists y \in Y, (y, z) \in \CALA \}$.
By the {\bf region dominated} by $Y$ we mean $Y.\REG = Y \cup Y.\NBHD$.\footnote
	{
	There is a large literature on dominating sets in undirected networks, $c.f.$
	\cite{HayHedSla98b,HayHedSla98}.
	}
Suppose $P$ is a set of individuals and the relation $\CALA$ denotes a 
relationship between them.
This relationship may be symmetric, such as ``mutual communication'', asymmetric, such as 
``hierarchical control'', or mixed such as ``friendship''.
The neighborhood $y.\NBHD$ about a person $y$ is the set of individuals with which $y$ 
directly relates.
The neighborhood, $Y.\NBHD$, of a set $Y$ of individuals is the set of individuals
not in $Y$ with whom at least one individual in $Y$ directly relates.
The region, $Y.\REG$ = $Y \cup Y.\NBHD$.
Members of $Y$ may, or may not, relate to each other.

We can visualize the
neighborhood structure of a discrete set of points, or individuals,  as
a graph such as Figure \ref{GRAPH}.\footnote
	{
	Whether $\CALA$ is regarded as reflexive, or not, is usually
	a matter of personal choice.
	The rows of a reflexive $\CALA$ capture the ``region'' concept;
	if irreflexive they represent the ``neighborhood'' concept.
	}
The neighbors of any point are those adjacent in the graph.
\begin{figure}[ht]
\centerline{\psfig{figure=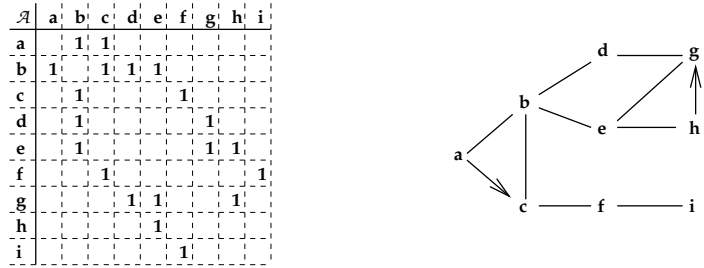,height=1.4in}}
\caption{A ``mixed'' adjacency matrix $\CALA$ and corresponding graph.
\label{GRAPH} }
\end{figure}
Thus, in the graph of Figure \ref{GRAPH} we have
$\{a\}.\NBHD = \{b, c\}$ or more simply
$a.\NBHD = bc$.
However, $a \not\in c.\NBHD = \{bf\}$.
Readily $g.\REG = deg$, and $h.\REG = egh$.
We may use the set delimiters, $\{ \ldots \}$, if we want to emphasize
its ``set nature''.

Given the neighborhood concepts $\NBHD$ and $\REG$, 
we define the {\bf neighborhood closure},
$\NCL$ to be
\begin{equation} \label{NCDEF}
Y.\NCL = \{ x | x.\REG \subseteq Y.\REG \}
\end{equation}
In a social system, the closure of a group $Y$ of individuals
are those additional individuals, $x$, all of whose connections
match those of the group $Y$.
A minimal set $X \subseteq Y$ of individuals for which 
$X.\NCL = Y.\NCL$ is sometimes called the nucleus, core, or 
generator of $Y.\NCL$.
Readily, for all $Y$,
\begin{equation} \label{NCCONT}
Y \subseteq Y.\NCL \subseteq Y.\REG
\end{equation}
\noindent
that is, $Y$ closure is always contained in the region dominated by $Y$.

Proofs of the following 3 propositions can be found in 
\cite{Pfa11}.
\begin{samepage}
\begin{proposition}\label{p.IDEMPOTENT}
$\NCL$ is a closure operator.
\end{proposition}
\end{samepage}

\begin{samepage}
\begin{proposition}\label{p.CONTAIN}
$X.\NCL \subseteq Y.\NCL$ if and only if $X.\REG \subseteq Y.\REG$.
\end{proposition}
\end{samepage}

\noindent
Readily, $X.\NCL = Y.\NCL$ if and only if $X.\REG = Y.\REG$.

\begin{samepage}
\begin{proposition}\label{p.HALO3}
Let $\NCL$ be the closure operator.
If $y.\NBHD \neq \EMPTYSET$ then there exists $X \subseteq y.\NBHD$ 
such that $y \in X.\NCL$.
\end{proposition}
\end{samepage}

\noindent
So, unless $y$ is an isolated point, every point $y$ is in the 
closure of some subset of its neighborhood.

One might expect that every point in a discrete network must
be closed, $e.g.$ $\{x\}.\NCL = \{x\}$. 
But, this need not be true, as can be seen in Figure \ref{GRAPH}.
The region $b.\REG = \{abcde\}$ while
$a.\REG = \{abc\} \subseteq b.\REG$,
so $b.\NCL = \{ab\}$.
Similarly, $e.\NCL = \{eh\}$.

\subsection{Separation and Connectivity} \label{SC}

Two sets $X$ and $Z$ are said to be {\bf separated} if 
$X.\REG\ \cap\ Z.\REG = \EMPTYSET$.
Similarly, two sets $X$ and $Z$ are {\bf connected} if 
$X.\REG\ \cap\ Z.\REG \neq \EMPTYSET$.
And a set $Y$ is said to be {\bf connected} if there do not
exist separated sets $X, Z$ such that $Y = X\ \cup\ Z$.

Separation, and connectivity, are defined in terms of dominated regions.
Proposition \ref{p.NSEP} recasts this in terms of neighborhoods.

\begin{samepage}
\begin{proposition}\label{p.NSEP}
$X$ and $Z$ are separated if and only if 
$X.\NBHD\ \cap\ Z.\NBHD = X\ \cap\ Z.\NBHD = X.\NBHD\ \cap\ Z = X\ \cap\ Z = \EMPTYSET$.
\end{proposition}
\end{samepage}
\noindent
\begin{proof}
Suppose $X, Z$ are not separated, then
$X.\REG\ \cap\ Z.\REG = (X \cup X.\NBHD)\ \cap\ (Z \cup Z.\NBHD) \neq \EMPTYSET$.
Consequently, either $X \cap Z$ or
$X \cap Z.\NBHD$ or $X.\NBHD \cap Z$ or $X.\NBHD \cap Z.\NBHD \neq \EMPTYSET$.
The converse is similar.
\qed
\end{proof}

When $\CALA$ is symmetric, Figure \ref{2C} illustrates 
two examples of disjoint, but connected, sets $X$ and $Z$ suggested by the condition of
Proposition \ref{p.NSEP}.
In Figure \ref{2C}(a) both $X.\NBHD\ \cap\ Z$ and $X\ \cap\ Z.\NBHD$ are
non-empty, while in (b) $X.\NBHD\ \cap\ Z = X\ \cap\ Z.\NBHD = \EMPTYSET$,
however $X.\NBHD \ \cap\ Z.\NBHD \neq \EMPTYSET$.
So $X$ and $Z$ are not separated.

\begin{figure}[ht]
\centerline{\psfig{figure=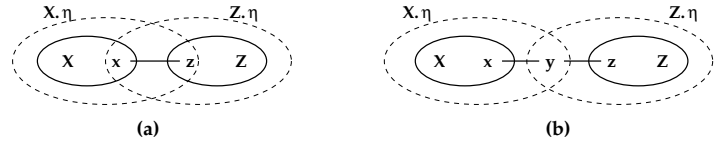,height=0.7in}}
\caption{Two examples of symmetric connectivity
\label{2C} }
\end{figure}

When $\CALA$ is not symmetric the possibilities of separation and connectivity
become more interesting.
In Figure \ref{3DC}(a), $X$ and $Z$ are not separated because $X.\NBHD\ \cap\ Z \neq \EMPTYSET$.
\begin{figure}[ht]
\centerline{\psfig{figure=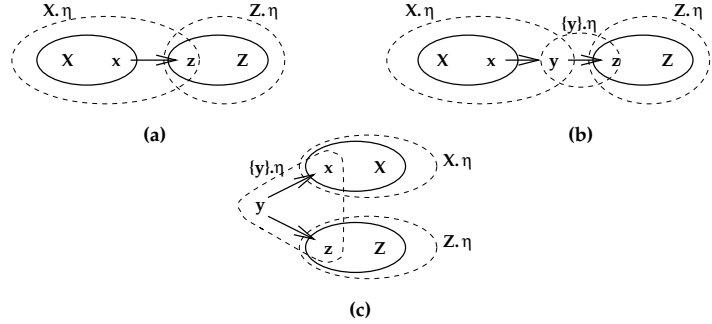,height=1.65in}}
\caption{Connectivity and separation when adjacency is asymmetric.
\label{3DC} }
\end{figure}
In Figure \ref{3DC}(b), $X$ and $Z$ are separated because $X.\NBHD \ \cap\ Z.\NBHD = \EMPTYSET$,
even though there is a path from $x$ to $z$ and $X.\NBHD \ \cap\ \{y\} \neq \EMPTYSET$ and
$\{y\}.\NBHD\ \cap\ Z \neq \EMPTYSET$.
Again in Figure \ref{3DC}(c), the sets $X$ and $Z$ are separated, although if 
the relationship $\CALA$ were symmetric they would not be.
In both (b) and (c) the set $X\ \cup\ \{y\}\ \cup\ Z$ is connected.

Readily, this definition of connectivity captures that of edge connectivity in graphs.

Monotone continuous transformations
do not preserve connectivity as do classical continuous functions.
In fact, we will show that in a {\it sufficiently large}
closure space, monotone, continuous transformations
preserve separation.
We must explain what we mean by {\it sufficiently large}.

Many idiosyncratic behaviors occur with small discrete examples.
For example, although we show in Proposition \ref{p.SEP} that continuous
transformations preserve separation, given the two small 
systems $\CALS$ and $\CALS'$ of Figure \ref{CX1}, this is manifestly
not true.
\begin{figure}[ht]
\centerline{\psfig{figure=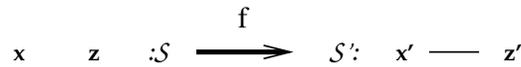,height=0.3in}}
\caption{$f$ is monotone, continuous, but does not preserve separation.
\label{CX1} }
\end{figure}
The only non-empty closed set of $\CALS'$ is $\{x'z'\}$, so of necessity
for any $Y$, $Y.\NCL.f \subseteq Y.f.\NCL'$.
There is insufficient ``structure'' around $x$, and $x'$, to be able to 
make general statements about the behavior of $f$.
We counter this by requiring in some propositions that the sets involved
be {\bf sufficiently large}.
This hugely simplifies the propositional statement, however in the proof,
we always specify just what we mean in this instance by ``sufficiently
large''.

\begin{samepage}
\begin{proposition}\label{p.SEP}
Let $\CALS$ be an atomistic space and let $f$ be monotone and continuous.
If $X$ and $Z$ are separated, then $X.f$ and $Z.f$ are separated,
provided $X$ and $Z$ are sufficiently large.
\end{proposition}
\end{samepage}
\noindent
\begin{proof}
Assume that $X.f$ and $Z.f$ are not separated, so
$X.f.\REG'\ \cap\ Z.f.\REG' \neq \EMPTYSET$.
By Prop. \ref{p.NSEP}, either $X.f \cap Z.f$ or $X.f.\NBHD' \cap Z.f$ or
$X.f \cap Z.f.\NBHD'$ or $X.f.\NBHD' \cap Z.f.\NBHD'$ is non-empty.
\\
By Prop. \ref{p.ATOM}, $X.f \cap Z.f \neq \EMPTYSET$ implies
$X \cap Z \neq \EMPTYSET$, contradicting the separation of $X$ and $Z$.
\\
Suppose $X.f.\NBHD' \cap Z.f \neq \EMPTYSET$,
implying $\exists z' \in Z.f$ such that $z' \in X.f.\NBHD'$
and $\exists x' \in X.f, z' \in x'.\NBHD'$.
Consider $x \in X, x = x'.f^{-1}$.
We may assume that $\{x\}.\NBHD \neq \EMPTYSET$ (since $X$ is sufficiently
large) and that by Prop. \ref{p.HALO3}, $x \in W.\NCL, W \subseteq \{x\}.\NBHD$.
Now, $x' \in W.\NCL.f$, but $x' \not\in W.f.\NCL'$ because $z' \not\in W.f.\NBHD'$
contradicting continuity.
\\
A similar argument holds if $X.f \cap Z.f.\NBHD' \neq \EMPTYSET$.
\\
Finally, we must consider the case where
$\exists y \in X.f.\NBHD' \cap Z.f.\NBHD' \neq \EMPTYSET$.
Possibly, $\{y'\}.f^{-1} = \EMPTYSET$.
However, we obtain the preceding contradiction of continuity by simply 
letting $y'$ take the role of $z'$.
\qed
\end{proof}

Proposition \ref{p.SEP} is a more general restatement of Proposition 15, found in
\cite{Pfa11}.
A weaker version can be demonstrated over non-atomistic ground sets if surjectivity
is assumed.

Entropy, in the sense that complex systems tend to break down into simpler
systems, seems to be reflected in Proposition \ref{p.SEP}.
Separation is preserved under ``smooth'', continuous change.
Creating connections (edges) is almost always a discontinuous process requiring
effort.
Breaking connections, however, is nearly always a continuous process, as shown
by the following.

\begin{samepage}
\begin{proposition}\label{p.DEL}
Let $\CALS$ be atomistic.
A monotone transformation $f$, which deletes a symmetric edge $(x, z)$
from $\CALA$ will be discontinuous if and only if either
\\
\hspace*{0.2in}
(a)
$z \in x.\NCL$ (or $x \in z.\NCL$), and $x.\NCL \neq z.\NCL$
\\
or
\\
\hspace*{0.2in}
(b)
$(x,z)$ is an edge in a chordless cycle $< v, \ldots, w, x, z, \ldots, v >$
\\
\hspace*{0.4in}
where either $|x.\NBHD| = 2$ or $|z.\NBHD| = 2$.
\end{proposition}
\end{samepage}
\begin{proof}
Suppose (a) holds and $z \in x.\NCL$.
Then $z' = z.f \in \{x\}.\NCL.f$, but $z' \not\in {x}.f.\NCL'$, so
$f$ is discontinuous.
\\
Suppose (b) holds and $|x.\NBHD| = 2$ , so $x.\NBHD = \{w, z\}$.
Now $x \in \{w, z\}.\NCL$ so $x' = \{x\}.f \in \{wz\}.\NCL.f$,
but $x' = x.f \not\in \{wz\}.f.\NCL'$ because $x \not\in \{w'z'\}.\NBHD'$
Again $f$ is discontinuous.
The reasoning is similar when $x.\NBHD = \{wz\}$.
\\
\\
Conversely, suppose $f$ is discontinuous.
Let $Y$ be a minimal set such that $Y.\NCL.f \not\subseteq Y.f.\NCL'$.
Readily, either $z \in Y.\NCL$ but $z' = z.f \not\in Y.f.\NCL'$ ,
(or $x \in Y.\NCL, x' = x.f \not\in Y.f.\NCL'$).
We may assume the former.
If $z \in Y$ then $z' \in Y.f.\NCL'$ trivially, so $z \in Y.\NBHD$.
Moreover, $z \in Y.\NCL$ implies $z.\NBHD \subseteq Y.\REG$.
Since $(x, z) \in \CALA$, $z \in x.\NBHD$, thus $x \in Y$.
If $Y = \{x\}$, then (a) holds and we are done.
\\
Assuming $z \not\in x.\NCL$ there must exist
$v \in z.\NBHD, v \not\in x.\NBHD$.
Since $z \in Y.\NCL$, $\exists w \in Y, v \in w.\NBHD$.
We claim this cycle $< v, w, x, z, v >$ is chordless.
$v \not\in x.\NBHD$ because $z \not\in x.\NCL$.
$z \not\in w.\NBHD$ because $Y$ is minimal.
\qed
\end{proof}

\noindent
Comments: Since to be a cycle, $z \in x.\NBHD$ and $x \in z.\NBHD$,
the condition of (b) above restricts either $x$ or $z$ go be of degree 2.
The second half of condition (a), $x.\NCL \neq z.\NCL$, is needed only for situations such
as that of Figure \ref{EXAMPLE}
\begin{figure}[ht]
\centerline{\psfig{figure=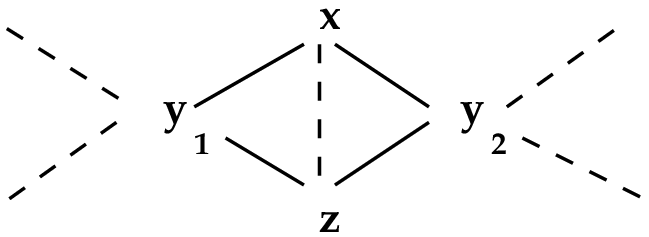,height=0.5in}}
\caption{Two points with $x.\NCL = z.\NCL$.
\label{EXAMPLE} }
\end{figure}
in which $x.\NCL = z.\NCL$ regardless of what
other nodes are connected to $y_1$ and $y_2$.
Addition, or deletion, of the dashed edge $(x, z)$ makes no change
in the closed set structure whatever.

In social terms, Proposition \ref{p.DEL} would assert that breaking a connection
between $x$ and $z$ represents a discontinuity if $z$ is very tightly bound to
$x$, that is has the same shared connections to others nearby.
This certainly seems consistent with the real world.
That breaking a chordless 4-cycle can be discontinuous is more surprising.

\subsection {Triadic Closure}

Creating a relationship, or edge $(x, z)$, will be continuous if $x$ and $z$ 
are already connected, that is, there exist $y \in x.\NBHD$ and $y \in z.\NBHD$.
The creation of $(x, z) \in \CALA$ is commonly known as
{\bf triadic closure}
The study of triads was initiated by Granovetter in
\cite{Gra73},
although he did not use the term ``closure''.
It is not truly a closure operator (it is not idempotent);
however, it
appears to be a frequently occurring process in dynamic social systems
\cite{HanPetDodWat07,MolVolFla12,Ops11}.
Kossinets and Watts 
\cite{KosWat06}
observe that ``For some specified value of $d_{ij}$, cyclic closure bias
is defined as the empirical probability that two previously unconnected 
individuals who are distance $d_{ij}$ apart in the network will initiate
a new tie.  
Thus cyclic closure naturally generalizes the notion of triadic closure''
(p. 88).

\subsection{Chordless $k$-Cycles} \label{CkC}

A closure system $\CALS$ is said to be {\bf irreducible} if every 
singleton set is closed.
The system of Figure \ref{GRAPH} is not irreducible because
$\{b\}.\NCL = \{ab\}$ and $\{f\}.\NCL = \{fi\}$.
We say the elements $a$ and $i$ are subsumed by $b$ and $f$
respectively because any closed set containing $b$, or $f$, much also
contain $a$, or $i$.

An iterative process which reduces any graph by successively
deleting these subsumed points (they contribute little to our 
understanding of the closed set structure) and their relationships
is described in 
\cite{Pfa11,Pfa12}.
There it is shown that a point $y$ will be in an irreducible 
subgraph if and only if $y$ is part of a chordless cycle of 
length $\geq$ 4, or on a path between two such chordless cycles.
These chordless cycles of length 4, or greater, we call
{\bf chordless $k$-cycles}, or just $k$-cycles.\footnote
	{
	A graph with no chordless cycles of length $\geq$ 4 is
	called a ``chordal graph''.
	Chordal graphs have an extensive literature, $c.f.$
	\cite{JacPet90,McK93}.
	}
The system of Figure \ref{GRAPH} contains just one chordless
$k$-cycle, $< b, d, g, e, b >$.
It is surmised that knowing the $k$-cycles of a system is one 
key to understanding its global structure.

The sets $X = \{egh\}$ and $Z = \{fi\}$ are separated in Figure \ref{GRAPH}.
Suppose a transformation $f$ connects them, either by creating a 
simple relationship/edge $(h, i)$ as shown in Figure \ref{GRAPH2},
or by adding a new point $y$ with $\{y\}.\NBHD = \{hi\}$.
\begin{figure}[ht]
\centerline{\psfig{figure=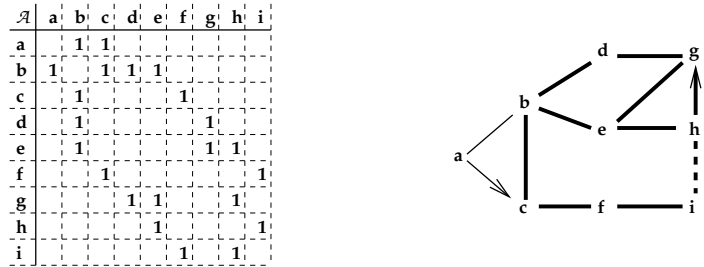,height=1.4in}}
\caption{The chordless $k$-cycles of Figure \ref{GRAPH} after adding the link $(h, i )$.
\label{GRAPH2} }
\end{figure}
Readily, the global structure, which now has two chordless $k$-cycles
is quite different.
They are $< b, d, g, e, b >$ and $< b, c, f, i, h, e, b >$.
Both have been emboldened in Figure \ref{GRAPH2}.\footnote
	{
	Each element of the triangle $< e h g e >$ is an element of one of the
	two $k$-cycles.
	}
In a sense, the change is ``high energy''.
The change in the matrix $\CALA$ is barely noticeable.
By Proposition \ref{p.SEP}, such a transformation must be
discontinuous.

``Entropy'', in the sense that complex, highly organized systems tend
to break down into simpler, more random systems seems to be
reflected in Proposition \ref{p.SEP}.
Separation is preserved under ``smooth'', continuous change.
Except for triadic closure, creating more connections (edges) is a 
discontinuous process requiring effort.
By Proposition \ref{p.DEL}, breaking connections, however, is most often a continuous
process.

Granovetter, in \cite{Gra73}, arrives at a somewhat analogous
conclusion.
He observes that
``the configuration of three strong ties became increasingly
frequent as people know one another longer and better''
[p. 1364], or equivalently, the ongoing social processes tend to
create triadic closure.
He also studied the configurations where two successive links
were not triadically closed.
He called these ``bridges'' and illustrated them in the following
Figure \ref{GRAPH3}, which I have re-drawn from his original.
\begin{figure}[ht]
\centerline{\psfig{figure=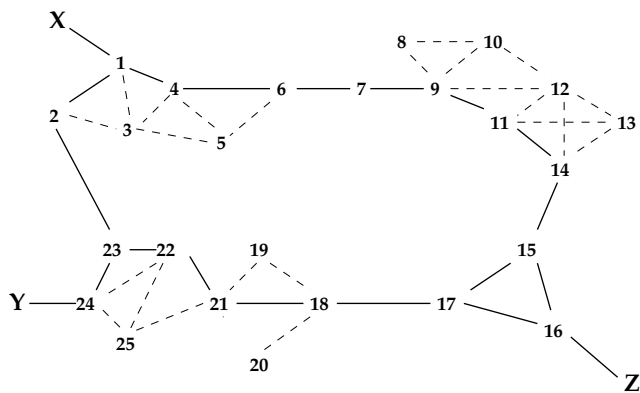,height=2.0in}}
\caption{A re-drawn version of Granovetter's Figure 2 (p.1365).
\label{GRAPH3} }
\end{figure}
$X$, $Y$ and $Z$ denote unrepresented, but connected subgraphs.
Granovetter was concerned with strong and weak ties between 
individuals and used his figure to illustrate his contention that
``no strong tie is a bridge''.
Our version is drawn to emphasize the large chordless $k$-cycle
(solid lines) and clusters of subsumed nodes (dashed lines).
No node in the triangle $\{ 15, 16, 17 \}$ is subsumed because
15 and 17 lie on the large $k$-cycle, and 16 presumably
lies on some path to a chordless $k$-cycle in the substructure labeled $Z$.
 
\section {Summary} \label{S}

Entropy in networks is real.
Systems do break down.
In this paper we have proposed a non-statistical model to describe
this process.
Smooth, continuous processes can remove relationships/edges throughout
the system, except those tightly bound in a closed cluster or chordless
4-cycle.
But, continuous processes cannot create any relationship, or link,
between separated subsets.
Proposition \ref{p.SEP} was a complete surprise.
We had predicted just the opposite;
that they would preserve connectivity as do graph homomorphisms
which are continuous.
Proposition \ref{p.SEP} and, to a lesser degree, Proposition
\ref{p.DEL} appear to have significant relevance to the behavior
of dynamic social networks.
They are the major contribution of this paper.

If these observations represent reality, the result of continuous 
change, or evolution, on a network
should be a collection of triadic clusters
loosely connected by bridges.
Since it is generated by ``entropy'', this kind of network might
be considered to be the epitome of a ``random'' network.

The author is not sufficiently well trained as a sociologist to 
assess the relevance of the mathematical approach developed in this
paper to the work of Granovetter and others.
But, it appears that we are actually describing the kinds
of networks that appear in sociology, and that concepts of continuity
and discontinuity based on closed sets are relevant.
What we need is to test these results against a number of large,
dynamic social networks.

{\small
\bibliography{/home/jlp/MATH.d/BIB.d/mathDB,/home/jlp/MATH.d/BIB.d/pfaltzDB,/home/jlp/MATH.d/BIB.d/otherDB,/home/jlp/MATH.d/BIB.d/orlandicDB,/home/jlp/MATH.d/BIB.d/haddletonDB,/home/jlp/MATH.d/BIB.d/databaseDB,/home/jlp/MATH.d/BIB.d/evsciDB,/home/jlp/MATH.d/BIB.d/aiDB,/home/jlp/MATH.d/BIB.d/closureDB,/home/jlp/MATH.d/BIB.d/dataminingDB,/home/jlp/MATH.d/BIB.d/conceptDB,/home/jlp/MATH.d/BIB.d/accessDB,/home/jlp/MATH.d/BIB.d/logicDB,/home/jlp/MATH.d/BIB.d/closureappDB,/home/jlp/MATH.d/BIB.d/gtransDB,/home/jlp/MATH.d/BIB.d/topologyDB,/home/jlp/MATH.d/BIB.d/socialDB}
}

\end{document}